\documentclass[a4paper, 11pt]{article}

\usepackage{fullpage}

\usepackage[english]{babel}
\usepackage[utf8]{inputenc}
\usepackage{amsmath, amsfonts, amssymb, amsthm}
\usepackage{mathtools}
\usepackage{bm}
\usepackage{authblk}
\usepackage{xfrac}
\usepackage[symbol]{footmisc}

\usepackage{booktabs}
\usepackage[colorlinks=true, linkcolor=blue, citecolor=red]{hyperref}
\usepackage{enumitem}
\usepackage{graphicx}
\usepackage{subcaption}

\usepackage[linesnumbered,ruled,vlined]{algorithm2e}
\SetKwInput{KwInput}{input}
\SetAlFnt{\small}

\usepackage[backend=biber, style=nature, sortcites=ynt, terseinits=true, maxbibnames=4, doi=true]{biblatex}

\DeclareFieldFormat{pages}{\mkcomprange{#1}}
\DeclareFieldFormat{pmlr}{%
  \mkbibacro{PMLR}\addcolon\space
  \ifhyperref
    {\href{https://proceedings.mlr.press/#1}{\nolinkurl{#1}}}
    {\nolinkurl{#1}}}

\DeclareFieldAlias{eprint:pmlr}{pmlr}
\DeclareFieldAlias{eprint:PMLR}{eprint:pmlr}

\DeclareFieldFormat{hal}{%
  \mkbibacro{HAL}\addcolon\space
  \ifhyperref
    {\href{https://hal.archives-ouvertes.fr/#1}{\nolinkurl{#1}}}
    {\nolinkurl{#1}}}

\DeclareFieldAlias{eprint:hal}{hal}
\DeclareFieldAlias{eprint:HAL}{eprint:hal}

\renewbibmacro*{eprint}{%
  \printfield{pmlr}%
  \printfield{hal}%
  \newunit\newblock
  \iffieldundef{eprinttype}
    {\printfield{eprint}}
    {\printfield[eprint:\strfield{eprinttype}]{eprint}}
}

\newcommand{\Real}{\mathbb{R}}
\newcommand{\N}{\mathbb{N}}

\newcommand{\set}[3][]{#1\{ #2~:~#3 #1\}}

\newcommand{\Proj}[3][]{\Pi_{#2} #1( #3 #1)}

\newcommand{\norm}[2][]{#1\| #2 #1\|}
\newcommand{\Norm}[3][]{#1\| #2 #1\|_{\mathrm{#3}}}

\newcommand{\rank}[2][]{\mathrm{rank} #1( #2 #1)}

\newcommand{\Expectation}{\mathbb{E}}
\newcommand{\Probability}{\Pr}

\newcommand{\Prob}[2][]{\Probability #1\{ #2 #1\}}

\newtheorem{theorem}{Theorem}[section]
\newtheorem{lemma}[theorem]{Lemma}
\newtheorem{corollary}[theorem]{Corollary}

\newtheorem{remark}[theorem]{Remark}

\title{On the distance to low-rank matrices in the maximum norm\footnotetext{This research was funded by the Austrian Science Fund (FWF) project \href{https://doi.org/10.55776/F65}{10.55776/F65}.}}
\author{Stanislav Budzinskiy\footnote{Faculty of Mathematics, University of Vienna, Kolingasse 14-16, 1090, Vienna, Austria (\href{mailto:stanislav.budzinskiy@univie.ac.at}{stanislav.budzinskiy@univie.ac.at})}}
\date{}

\begin{document}
\maketitle

\begin{abstract}
Every sufficiently big matrix with small spectral norm has a nearby low-rank matrix if the distance is measured in the maximum norm (Udell \& Townsend, SIAM J Math Data Sci, 2019). We use the Hanson--Wright inequality to improve the estimate of the distance for matrices with incoherent column and row spaces. In numerical experiments with several classes of matrices we study how well the theoretical upper bound describes the approximation errors achieved with the method of alternating projections.
\end{abstract}

Hanson--Wright inequality, alternating projections\\



\section{Introduction}

To check whether a matrix $X \in \Real^{m \times n}$ admits \textit{good} low-rank approximation, we typically look at its singular values $\sigma_1, \ldots, \sigma_{\min\{m,n\}}$. The theorems of Schmidt \cite{schmidt1907theorie}, Eckart-Young \cite{eckart1936approximation} and Mirsky \cite{mirsky1960symmetric} state that for every unitarily invariant norm $\Norm{\cdot}{\circ}$ and every $r < \rank{X}$ the minimum of the functional $\Norm{X - Y}{\circ}$ over all matrices $Y$ with $\rank{Y} = r$ equals 
\begin{equation*}
    \Norm[\big]{\mathrm{diag}\big(0, \ldots, 0, \sigma_{r+1}, \ldots, \sigma_{\min\{m,n\}}\big)}{\circ}
\end{equation*}
and is achieved at the truncated singular value decomposition of $X$. Recall that a matrix norm is called unitarily invariant if $\Norm{Q_m A Q_n}{\circ} = \Norm{A}{\circ}$ for each $A \in \Real^{m \times n}$ and all orthogonal matrices $Q_m \in \Real^{m \times m}$ and $Q_n \in \Real^{n \times n}$. This class includes the widely used spectral $\Norm{\cdot}{2}$ and Frobenius $\Norm{\cdot}{F}$ norms, but there are other important norms that are not unitarily invariant and, therefore, are not covered by the above theorems. We focus on the \textit{maximum} norm
\begin{equation*}
    \Norm{X}{\max} = \max_{i,j} |X(i, j)|,
\end{equation*}
which is not unitarily invariant as a simple example shows:
\begin{equation*}
    \Norm[\Bigg]{\begin{bmatrix}
        \sfrac{1}{\sqrt{2}} & -\sfrac{1}{\sqrt{2}} \\
        \sfrac{1}{\sqrt{2}} & \sfrac{1}{\sqrt{2}}
    \end{bmatrix}}{\max} \neq 
    \Norm[\Bigg]{\begin{bmatrix}
        1 & 0 \\
        0 & 1
    \end{bmatrix}}{\max}.
\end{equation*}
Just as for any other norm, there is an ultimate upper bound on the approximation error that is achieved if we take a vanishing sequence of rank-$r$ matrices:
\begin{equation}
\label{eq:ultimate_bound}
    d_r(X) = \inf \set[\big]{\Norm{X - Y}{\max}}{\rank{Y} = r} \leq \Norm{X}{\max}.
\end{equation}

The ideas of \textit{cross approximation} \cite{goreinov2001maximal, osinsky2018pseudo} lead to the following bound
\begin{equation}
\label{eq:cross_bound}
    d_r(X) \leq \sqrt{1 + \frac{r}{m - r + 1}} \sqrt{1 + \frac{r}{n - r + 1}} \cdot \sigma_{r+1},
\end{equation}
which can be simplified to $d_r(X) \leq 2 \sigma_{r+1}$ when $m, n \geq 2r-1$. This result is similar in flavor to the case of the spectral norm but, unlike it, can become trivial. Indeed, the \textit{spikiness} of $X$,
\begin{equation*}
    \gamma_X = \sqrt{mn} \Norm{X}{\max} / \Norm{X}{2}, \quad 1 \leq \gamma_X \leq \sqrt{mn},
\end{equation*}
allows us to rewrite the above estimate as
\begin{equation*}
    d_r(X) \leq \frac{2\sqrt{mn}}{\gamma_X} \frac{\sigma_{r+1}}{\Norm{X}{2}} \Norm{X}{\max}
\end{equation*}
and note that the right hand side exceeds $\Norm{X}{\max}$ if $\sigma_{r+1} > \frac{1}{2} \Norm{X}{2}$. Therefore, the upper bound \eqref{eq:cross_bound} guarantees good approximation in the maximum norm when the singular values decay rapidly but ceases to be restrictive when the spectrum is flat.

A completely different route was taken in \cite{alon2013approximate}, where the authors showed that $d_r(X)$ can be small regardless of the singular values. Namely, they proved the following.

\begin{theorem}[{\cite[Theorem~1.4]{alon2013approximate}}]
\label{theorem:spd_max_appr}
Let $X \in \Real^{n \times n}$ be symmetric positive semidefinite. For $\varepsilon \in (0,1)$, there exists a matrix $Y \in \Real^{n \times n}$ such that
\begin{equation*}
    \Norm{X - Y}{\max} \leq \varepsilon \Norm{X}{max}, \quad \rank{Y} \leq \frac{9 \log(n)}{\varepsilon^2 - \varepsilon^3}.
\end{equation*}
\end{theorem}

\noindent For example, if $X$ is a multiple of the $n \times n$ identity matrix for sufficiently large $n$ then
\begin{equation*}
    d_{r}(X) \leq \varepsilon \Norm{X}{\max}, \quad r = \rank{Y},
\end{equation*}
while the best rank-$r$ approximation errors in the spectral and Frobenius norms are equal to $\Norm{X}{2}$ and $\sqrt{1 - \frac{r}{n}} \Norm{X}{F}$. To achieve the same relative approximation error $\varepsilon$ in the Frobenius norm, the rank must be at least $n(1 - \varepsilon^2)$, which scales linearly with $n$. For the spectral norm, the relative approximation error is always equal to one. At the same time, the rank scales as $\log(n)$ for the maximum norm---a truly low-rank approximation when $n$ is extremely large.

Theorem~\ref{theorem:spd_max_appr} was extended to nonsymmetric matrices in {\cite{udell2019big}}, though the maximum norm in the right hand side had to be replaced with the spectral norm.

\begin{theorem}[{\cite[Theorem~1.0]{udell2019big}}]
\label{theorem:ut_max_appr}
Let $X \in \Real^{m \times n}$. For $\varepsilon \in (0,1)$, there exists a matrix $Y \in \Real^{m \times n}$ such that
\begin{equation*}
    \Norm{X - Y}{\max} \leq \varepsilon \Norm{X}{2}, \quad \rank{Y} \leq \Big\lceil 72 \frac{\log(2 \min\{m,n\}+1)}{\varepsilon^2} \Big\rceil.
\end{equation*}
\end{theorem}
\noindent When $m$ and $n$ are sufficiently big, this theorem gives an upper bound
\begin{equation*}
    d_{r}(X) \leq \frac{\varepsilon \sqrt{mn}}{\gamma_X} \Norm{X}{\max}, \quad r = \rank{Y},
\end{equation*}
that is restrictive for $\varepsilon < \gamma_X / \sqrt{mn}$ even when all singular values of $X$ are equal. For larger $\varepsilon$, the bound exceeds the ultimate upper bound \eqref{eq:ultimate_bound} and becomes trivial.

\subsection{Contributions and outline}
Theorem~\ref{theorem:ut_max_appr} is an interesting statement about low-rank approximation in non-unitarily-invariant norms, but its proof in \cite{udell2019big} contains a small mistake. We feel obliged to correct it and simultaneously refine the estimate for matrices that are not full rank and whose column and row spaces are incoherent; this is the purpose of Theorem~\ref{theorem:our_max_appr}.

Following \cite{udell2019big}, we rely on the Johnson--Lindenstrauss lemma in the proof of Theorem~\ref{theorem:our_max_appr}. We turn to the Hanson--Wright inequality to prove Theorem~\ref{theorem:our2_max_appr}: an estimate of the approximation error that is tighter than Theorem~\ref{theorem:our_max_appr}.

Both Theorem~\ref{theorem:our_max_appr} and Theorem~\ref{theorem:our2_max_appr} provide upper bounds for the distance $d_r(X)$ in the maximum norm from $X$ to the set of rank-$r$ matrices. How accurate are they? We address this question numerically with the method of alternating projections as developed in \cite[Section~7.5]{budzinskiy2023quasioptimal}. As far as we know, this is the first extensive numerical study of low-rank matrix approximation in the maximum norm.

The text is naturally divided into three parts. In Section~\ref{sec:jl}, we recall the Johnson--Lin\-den\-strauss lemma and prove Theorem~\ref{theorem:our_max_appr}. We introduce the Hanson--Wright inequality in Section~\ref{sec:hw} and prove Theorem~\ref{theorem:our2_max_appr} with its help. Finally, we present the method of alternating projections and discuss the results of our numerical experiments in Section~\ref{sec:numerical}.
\section{Proof based on the Johnson--Lindenstrauss lemma}
\label{sec:jl}

The lemma of Johnson and Lindenstrauss \cite{johnson1984extensions} states that every finite set in a high-dimensional Euclidean space can be mapped into a lower-dimensional Euclidean space so that the pairwise distances between the points in the set are approximately preserved. Below, we present a particular version of the lemma from \cite{dasgupta2003elementary}. The $\ell_2$-norm of a vector is written as $\norm{\cdot}$.

\begin{theorem}[{\cite[Theorem~2.1]{dasgupta2003elementary}}]
\label{theorem:JL}
Let $\varepsilon \in (0,1)$. For every finite set $\{ w_1, \ldots, w_m \} \subset \Real^k$ with
\begin{equation*}
    k \geq r \geq \frac{4 \log(m)}{\varepsilon^2 / 2 - \varepsilon^3 / 3} 
\end{equation*}
there exists a matrix $Q \in \Real^{k \times r}$ with orthonormal columns such that for all $i,j \in [m]$
\begin{equation*}
    (1 - \varepsilon) \norm[\big]{w_i - w_j}^2 \leq \frac{k}{r} \norm[\big]{Q^T (w_i - w_j)}^2 \leq (1 + \varepsilon) \norm[\big]{w_i - w_j}^2.
\end{equation*}
\end{theorem}

The Johnson--Lindenstrauss lemma plays a pivotal role in the proof of Theorem~\ref{theorem:ut_max_appr} in \cite{udell2019big}. The lemma is not applied directly, though; rather, it is used to show that the pairwise \textit{inner products} are approximately preserved together with the pairwise distances.

\begin{corollary}
\label{corollary:JL}
Let $\varepsilon \in (0,1)$. For every finite set $\{ w_1, \ldots, w_m \} \subset \Real^k$ with
\begin{equation*}
    k \geq r \geq \frac{4 \log(m+1)}{\varepsilon^2 / 2 - \varepsilon^3 / 3}
\end{equation*}
there exists a matrix $Q \in \Real^{k \times r}$ with orthonormal columns such that for all $i,j \in [m]$
\begin{equation*}
    \left| w_i^T w_j - \frac{k}{r} (Q^T w_i)^T (Q^T w_j) \right| \leq \varepsilon \left( \norm{w_i}^2 + \norm{w_j}^2 - w_i^T w_j \right).
\end{equation*}
\end{corollary}
\begin{proof}
    Follow the proof of \cite[Lemma~2.4]{udell2019big} and use Theorem~\ref{theorem:JL}.
\end{proof}

To formulate our Theorem~\ref{theorem:our_max_appr}, a refined version of Theorem~\ref{theorem:ut_max_appr}, we need to introduce the \textit{coherence of a subspace}. Let $L$ be a $k$-dimensional subspace of $\Real^m$ and $\Pi_L : \Real^{m} \to L$ be the corresponding orthogonal projection operator. The coherence of $L$ is defined as the largest squared norm of the canonical basis vectors projected onto $L$:
\begin{equation*}
    \mu_L = \frac{m}{k} \max_{i \in [m]} \norm{\Pi_L e_i}^2.
\end{equation*}
For every subspace $L$, its coherence $\mu_L$ lies in the closed interval $[1, m/k]$.

\begin{theorem}
\label{theorem:our_max_appr}
Assume that $m, n, k \in \N$ satisfy
\begin{equation*}
    \min\{m,n\} \geq k > k_0 = 108 \log(m + n + 1).
\end{equation*}
Let $\varepsilon \in (\sqrt{k_0 / k}, 1)$ and assume that $r \in \N$ satisfies $k > r \geq k_0 / \varepsilon^2$. Then for every matrix $X \in \Real^{m \times n}$ of rank $k$ there exists a matrix $Y \in \Real^{m \times n}$ of rank $r$ such that
\begin{equation*}
    \Norm{X - Y}{\max} \leq \frac{\varepsilon}{3} \left(\frac{k}{m} \mu_{\mathrm{col}} + \frac{k}{n} \mu_{\mathrm{row}} + \frac{\gamma_X}{\sqrt{mn}} \right) \Norm{X}{2},
\end{equation*}
where $\mu_{\mathrm{col}}$ and $\mu_{\mathrm{row}}$ are the coherences of the column and row spaces of $X$.
\end{theorem}
\begin{proof}
Let $X = U \Sigma V^T$ be the thin singular value decomposition of $X$ with $U \in \Real^{m \times k}$, $\Sigma \in \Real^{k \times k}$ and $V \in \Real^{n \times k}$. Consider matrices $\Tilde{U} = U \Sigma^{\frac{1}{2}}$ and $\Tilde{V} = V \Sigma^{\frac{1}{2}}$ and denote their rows as
\begin{equation*}
    \Tilde{U}^T = [ \Tilde{u}_1 \dots \Tilde{u}_m ], \quad \Tilde{V}^T = [ \Tilde{v}_1 \dots \Tilde{v}_n ].
\end{equation*}
We then apply Corollary~\ref{corollary:JL} with $\Tilde{\varepsilon} = \varepsilon/3$ to the set $\{ \Tilde{u}_1, \ldots \Tilde{u}_m, \Tilde{v}_1, \ldots, \Tilde{v}_n \} \subset \Real^k$. The choice of $r$ guarantees that
\begin{equation*}
    k \geq r \geq 108\frac{ \log(m+n+1)}{\varepsilon^2} \geq \frac{4 \log(m+n+1)}{\Tilde{\varepsilon}^2 / 2 - \Tilde{\varepsilon}^3 / 3},
\end{equation*}
so there is a matrix $Q \in \Real^{k \times r}$ with orthonormal columns such that for all $i \in [m]$ and $j \in [n]$
\begin{equation*}
    \left| \Tilde{u}_i^T \Tilde{v}_j - \frac{k}{r} (Q^T \Tilde{u}_i)^T (Q^T \Tilde{v}_j) \right| \leq \Tilde{\varepsilon} \left( \norm{\Tilde{u}_i}^2 + \norm{\Tilde{v}_j}^2 - \Tilde{u}_i^T \Tilde{v}_j \right).
\end{equation*}
Recall that $X(i,j) = \Tilde{u}_i^T \Tilde{v}_j$. Setting $Y = \frac{k}{r} (\Tilde{U} Q) (\Tilde{V} Q)^T$, we can rewrite the above inequality as
\begin{equation*}
    |X(i,j) - Y(i,j)| \leq \Tilde{\varepsilon} \left( \norm{\Tilde{u}_i}^2 + \norm{\Tilde{v}_j}^2 + \Norm{X}{\max} \right).
\end{equation*}
Next, we estimate the norms of the rows:
\begin{equation*}
    \norm{\Tilde{u}_i}^2 = \norm{\Sigma^{\frac{1}{2}} u_i}^2 \leq \Norm{\Sigma}{2} \cdot \norm{u_i}^2 = \Norm{X}{2} \cdot \norm{u_i}^2 \leq \frac{k}{m} \mu_{\mathrm{col}} \Norm{X}{2}.
\end{equation*}
Similarly, we have $\norm{\Tilde{v}_j}^2 \leq \frac{k}{n} \mu_{\mathrm{row}} \Norm{X}{2}$. Using the definition of spikiness we arrive at
\begin{equation*}
    \Norm{X - Y}{\max} \leq \frac{\varepsilon}{3} \left(\frac{k}{m} \mu_{\mathrm{col}} + \frac{k}{n} \mu_{\mathrm{row}} + \frac{\gamma_X}{\sqrt{mn}} \right) \Norm{X}{2}.
\end{equation*}
It remains to note that $\rank{Y} = r$ since $\Tilde{U}$, $\Tilde{V}$ and $Q$ are full rank.
\end{proof}

\begin{remark}
For $\varepsilon \leq \sqrt{k_0 / k}$, the actual rank $k$ and the approximation rank $r$ satisfy $r \geq k$, so that Theorem~\ref{theorem:our_max_appr} does not guarantee any compression.
\end{remark}

\begin{remark}
Since $\mu_{\mathrm{col}} \leq \frac{m}{k}$, $\mu_{\mathrm{row}} \leq \frac{n}{k}$ and $\gamma_X \leq \sqrt{mn}$, our Theorem~\ref{theorem:our_max_appr} guarantees that
\begin{equation*}
    \Norm{X - Y}{\max} \leq \varepsilon \Norm{X}{2}.
\end{equation*}
This is exactly the statement of Theorem~\ref{theorem:ut_max_appr}. In its original proof \cite[Theorem~1.0]{udell2019big}, the rows and columns of $\tilde{U}$ and $\tilde{V}$ were confused, which explains why the estimate of the rank in Theorem~\ref{theorem:ut_max_appr} depends on $2\min\{m,n\}$ rather than on $m + n$. All the other proofs in \cite{udell2019big} that follow the same line of reasoning are void of this inaccuracy.
\end{remark}

Theorem~\ref{theorem:our_max_appr} controls the approximation error better than Theorem~\ref{theorem:ut_max_appr} for matrices with low spikiness and incoherent column and row spaces. Consider two examples. First, the Hadamard matrix of size $n \times n$ has $\gamma_X = \sqrt{n}$ and $\mu_{\mathrm{col}} = \mu_{\mathrm{row}} = 1$ so that Theorem~\ref{theorem:our_max_appr} gives
\begin{equation*}
    \Norm{X - Y}{\max} \leq \frac{\varepsilon}{3} \left(2 + \frac{1}{\sqrt{n}}\right) \Norm{X}{2}.
\end{equation*}
The second example involves randomness. It was proved in \cite[Lemma~2.2]{candes2009exact} that if $L$ is a random $k$-dimensional subspace of $\Real^{m}$ then its coherence is with high probability bounded by
\begin{equation*}
    \mu_L \leq \tilde{C} \frac{\max\{ k, \log(m) \}}{k}.
\end{equation*}
Then if $X$ is designed to have random column and row spaces, Theorem~\ref{theorem:our_max_appr} says that
\begin{equation*}
    \Norm{X - Y}{\max} \leq \frac{\varepsilon}{3} \left( 1 + \frac{k}{m}\tilde{C} + \frac{k}{n}\tilde{C} \right) \Norm{X}{2}.
\end{equation*}
In both cases, even if $m$ and $n$ become arbitrarily large, the gain over Theorem~\ref{theorem:ut_max_appr} is within a constant factor. With the upcoming Theorem~\ref{theorem:our2_max_appr}, we achieve a more substantial improvement.

\section{Proof based on the Hanson--Wright inequality}
\label{sec:hw}

Under the hood, the Johnson--Lindenstrauss lemma is a probabilistic statement about the concentration of random variables. So is the Hanson--Wright inequality \cite{hanson1971bound, rudelson2013hanson} that establishes the concentration of a quadratic form in sub-Gaussian random variables around its mean. Recall that a real-valued random variable $\xi$ is called sub-Gaussian if
\begin{equation*}
    \Norm{\xi}{\psi_2} = \sup \set[\Big]{p^{-\frac{1}{2}} (\Expectation|\xi|^p)^{\frac{1}{p}}}{p \in \N} < \infty.
\end{equation*}

\begin{theorem}[{\cite[Theorem~1.1]{rudelson2013hanson}}]
\label{theorem:HW}
Let $A \in \Real^{n \times n}$. Consider a random vector $x \in \Real^{n}$ with independent components that satisfy $\Expectation [x(i)] = 0$ and $\Norm{x(i)}{\psi_2} \leq L$ for all $i \in [n]$. Then, for every $t \geq 0$, 
\begin{equation*}
    \Prob[\big]{| x^T A x - \Expectation [x^T A x] | \geq t} \leq 2 \exp\left(  -\frac{1}{c} \min\left\{ \frac{t^2}{L^4 \Norm{A}{F}^2}, \frac{t}{L^2 \Norm{A}{2}} \right\}  \right),
\end{equation*}
where $c > 0$ is an absolute constant.
\end{theorem}

In the proof of Theorem~\ref{theorem:our_max_appr}, we used Corollary~\ref{corollary:JL} to control the deviation of inner products from their means; our current plan is to employ the Hanson--Wright inequality for the task. To do so, we need to set up a specific quadratic form and study its properties. We write $\otimes$ for the Kronecker product of matrices and $I_r$ for the $r \times r$ identity matrix.

\begin{lemma}
Let $u, v \in \Real^k$ and $r \in \N$. Then the matrix $A = I_{r} \otimes u v^T \in \Real^{kr \times kr}$ satisfies
\begin{equation*}
    \Norm{A}{2} = \norm{u} \cdot \norm{v}, \quad \Norm{A}{F} = \sqrt{r} \cdot \norm{u} \cdot \norm{v}.
\end{equation*}
\end{lemma}
\begin{proof}
    The spectral and Frobenius norms of a Kronecker product of matrices are equal to the products of the respective norms of the matrices.
\end{proof}

\begin{lemma}
\label{lemma:special_hw}
Let $u, v \in \Real^k$, $r \in \N$ and $A = I_{r} \otimes u v^T \in \Real^{kr \times kr}$. Consider a random vector $x \in \Real^{kr}$ with independent components that satisfy $\Expectation [x(i)] = 0$ and $\Norm{x(i)}{\psi_2} \leq L$ with $L = r^{-\frac{1}{2}}$ for all $i \in [kr]$. Then, for every $\varepsilon \in [0,1]$, 
\begin{equation*}
    \Prob[\big]{| x^T A x - \Expectation [x^T A x] | \geq \varepsilon \cdot \norm{u} \cdot \norm{v} } \leq 2 \exp\left(  -\frac{r}{c} \varepsilon^2 \right).
\end{equation*}
\end{lemma}
\begin{proof}
The constant $L$ is chosen so that
\begin{equation*}
    L^4 \Norm{A}{F}^2 = \frac{1}{r} \norm{u}^2 \cdot \norm{v}^2, \quad L^2 \Norm{A}{2} = \frac{1}{r} \norm{u} \cdot \norm{v}.
\end{equation*}
Then Theorem~\ref{theorem:HW} guarantees that
\begin{equation*}
    \Prob[\big]{| x^T A x - \Expectation [x^T A x] | \geq t} \leq 2 \exp\left(  -\frac{r}{c} \min\left\{ \frac{t^2}{\norm{u}^2 \cdot \norm{v}^2}, \frac{t}{\norm{u} \cdot \norm{v}} \right\}  \right).
\end{equation*}
The squared term dominates whenever $t = \varepsilon \cdot \norm{u} \cdot \norm{v}$ with $\varepsilon \in [0,1]$.
\end{proof}

\begin{theorem}
\label{theorem:our2_max_appr}
Assume that $m, n, k \in \N$ satisfy, with an absolute constant $C > 0$,
\begin{equation*}
    \min\{m,n\} \geq k \geq k_0 = C \log(4mn).
\end{equation*}
Let $\varepsilon \in [\sqrt{k_0 / k}, 1]$ and assume that $r \in \N$ satisfies $k \geq r \geq k_0 / \varepsilon^2$. Then for every matrix $X \in \Real^{m \times n}$ of rank $k$ there exists a matrix $Y \in \Real^{m \times n}$ of rank $r$ such that
\begin{equation*}
    \Norm{X - Y}{\max} \leq \varepsilon \frac{k}{\sqrt{mn}} \sqrt{\mu_{\mathrm{col}} \mu_{\mathrm{row}}} \cdot \Norm{X}{2},
\end{equation*}
where $\mu_{\mathrm{col}}$ and $\mu_{\mathrm{row}}$ are the coherences of the column and row spaces of $X$.
\end{theorem}
\begin{proof}
Take a non-zero centered sub-Gaussian random variable $\xi$ with $\Norm{\xi}{\psi_2} \leq 1$ and populate a matrix $Q \in \Real^{k \times r}$ with independent copies of $r^{-\frac{1}{2}}\xi$. Following the proof of Theorem~\ref{theorem:our_max_appr}, define the matrices $\Tilde{U} \in \Real^{m \times k}$ and $\Tilde{V} \in \Real^{n \times k}$ with rows $\{ \Tilde{u}_i^T \}_{i = 1}^{m}$ and $\{ \Tilde{v}_j^T \}_{j = 1}^{n}$. Then we have
\begin{equation*}
    \Tilde{u}_i^T \Tilde{v}_j = X(i,j), \quad \Expectation[\Tilde{u}_i^T Q Q^T \Tilde{v}_j] = \Expectation|\xi|^2 \cdot X(i,j)
\end{equation*}
for every $i \in [m]$ and $j \in [n]$. Next, we can rewrite $\Tilde{u}_i^T Q Q^T \Tilde{v}_j$ as a quadratic form
\begin{equation*}
    \Tilde{u}_i^T Q Q^T \Tilde{v}_j = \mathrm{vec}(Q)^T \cdot [I_{r} \otimes \Tilde{u}_i \Tilde{v}_j^T] \cdot \mathrm{vec}(Q),
\end{equation*}
where $\mathrm{vec}(Q) \in \Real^{kr}$ is obtained by stacking the columns of $Q$, and use Lemma~\ref{lemma:special_hw} to get
\begin{equation*}
    \Prob[\Big]{\Big| \Tilde{u}_i^T Q Q^T \Tilde{v}_j - \Expectation|\xi|^2 \cdot X(i,j) \Big| \geq \Tilde{\varepsilon} \cdot \norm{\Tilde{u}_i} \cdot \norm{\Tilde{v}_j} } \leq 2 \exp\left(  -\frac{r}{c} \Tilde{\varepsilon}^2 \right), \quad \Tilde{\varepsilon} \in [0,1].
\end{equation*}
Assume that $r \geq c \log(4mn) / \Tilde{\varepsilon}^2$. Then the union bound guarantees that
\begin{equation*}
    \Prob[\Big]{\Big| \Tilde{u}_i^T Q Q^T \Tilde{v}_j - \Expectation|\xi|^2 \cdot X(i,j) \Big| \leq \Tilde{\varepsilon} \cdot \norm{\Tilde{u}_i} \cdot \norm{\Tilde{v}_j} \text{ for all } i \in [m], j \in [n]} \geq \frac{1}{2}.
\end{equation*}
Thus, the desired $Q$ exist, and we can select one of full rank. Set $Y = \frac{1}{\Expectation|\xi|^2} (\Tilde{U} Q) (\Tilde{V} Q)^T$; then
\begin{equation*}
    |X(i,j) - Y(i,j)| \leq \frac{\Tilde{\varepsilon}}{\Expectation|\xi|^2} \norm{\Tilde{u}_i} \norm{\Tilde{v}_j} \leq \varepsilon \frac{k}{\sqrt{mn}} \sqrt{\mu_{\mathrm{col}} \mu_{\mathrm{row}}} \cdot \Norm{X}{2}
\end{equation*}
with $\varepsilon = \Tilde{\varepsilon} / \min\{1, \Expectation|\xi|^2\}$ for all $i \in [m]$ and $j \in [n]$. Choosing $C = c / \min\{1, \Expectation|\xi|^2\}^2$, we have
\begin{equation*}
    k \geq r \geq C\frac{ \log(4mn)}{\varepsilon^2} = \frac{ c \log(4mn)}{\varepsilon^2 \min\{1, \Expectation|\xi|^2\}^2} = c\frac{ \log(4mn)}{\Tilde{\varepsilon}^2}, \quad \Tilde{\varepsilon} \in [0,1]. \qedhere
\end{equation*}
\end{proof}

Let us return to the example from Section~\ref{sec:jl}. Consider a rank-$k$ matrix $X \in \Real^{n \times n}$ whose column and row spaces are drawn at random; under the assumptions of Theorem~\ref{theorem:our2_max_appr}, for sufficiently large $n$, the coherences $\mu_{\mathrm{col}}$ and $\mu_{\mathrm{row}}$ of $X$ are bounded by an absolute constant $\Tilde{C}$ and there exists a matrix $Y$ of rank $C \log(4n^2) / \varepsilon^2$ such that
\begin{equation*}
    \Norm{X - Y}{\max} \leq \varepsilon \Tilde{C} \frac{k}{n} \Norm{X}{2}.
\end{equation*}
Take a sequence of such $n \times n$ matrices $\{ X_{n} \}$ with $\Norm{X_n}{2} \leq n^{\frac{\delta}{2}}$ and $\rank{X_{n}} = \lceil n^{1 - \delta} \rceil$ for some $\delta \in (0, 1)$. Then Theorem~\ref{theorem:our2_max_appr} guarantees that $d_{r_n}(X_n) \to 0$ with $r_n = C \log(4n^2) / \varepsilon^2$ as $n \to \infty$. This behavior cannot be explained with Theorem~\ref{theorem:our_max_appr}.

Theorem~\ref{theorem:our2_max_appr} still has limitations, though. For a sequence of random orthogonal matrices $\{ X_n \}$, it guarantees that $d_{r_n}(X_n) \leq \varepsilon$ for large $n$. Meanwhile, for every \textit{fixed} rank $r$, the ultimate upper bound \eqref{eq:ultimate_bound} gives $d_r(X_n) \leq \Norm{X_n}{\max}$, and it follows that $d_r(X_n) \to 0$ as $n \to \infty$ since $\sqrt{n / \log(n)} \cdot \Norm{X_n}{\max} \to 2$ in probability \cite{jiang2005maxima}.
\section{Numerical experiments}
\label{sec:numerical}

\subsection{Known algorithms}
Even though the maximum-norm distance from a matrix to the set of rank-$r$ matrices is always attained, we cannot expect to compute the closest matrix with a reasonably efficient numerical algorithm since the problem is known to be NP-hard even for $r = 1$ (unless the sign pattern of the minimizer is known, in which case there is an algorithm of polynomial complexity that checks if the distance is smaller than a given constant \cite{gillis2019low}). Combinatorial arguments were used in \cite{morozov2023optimal} to design an optimization algorithm that converges to the best rank-one approximant and is based on repeated alternating minimization with different initial conditions.

The rank-$r$ approximation problem was addressed in \cite{zamarashkin2022best} with alternating minimization (but the algorithm was only shown to reach local minimizers) and in \cite{gillis2019low} with a heuristic coordinate descent approach. The algorithm from \cite{chierichetti2017algorithms} is based on the column subset selection problem and is proved to find a rank-$r$ matrix $Y$ such that $\Norm{X - Y}{\max} = \mathcal{O}(r) \cdot d_r(X)$ using $\mathcal{O}(r^r \log^r(m) \mathrm{poly}(mn))$ operations. In \cite{kyrillidis2018simple}, gradient descent was applied to a smoothed and regularized functional; with an optimal (but practically impossible to select) choice of parameters, the algorithm can find a rank-$r$ approximant that satisfies $\Norm{X - Y}{\max} \leq (1 + \varepsilon) \cdot d_r(X)$ after $\mathcal{O}(\varepsilon^{-2})$ iterations.

\subsection{Alternating projections}
In this article, we shall use the method of alternating projections to numerically compute low-rank approximations in the maximum norm. Following \cite[Section~7.5]{budzinskiy2023quasioptimal}, we introduce the closed maximum-norm ball of radius $\varepsilon$ centered at $X$,
\begin{equation*}
    \mathbb{B}_\varepsilon(X) = \set{Z \in \Real^{m \times n}}{\Norm{X - Z}{\max} \leq \varepsilon},
\end{equation*}
and the set of rank-$r$ matrices
\begin{equation*}
    \mathcal{M}_{r} = \set{Y \in \Real^{m \times n}}{\rank{Y} = r}.
\end{equation*}
The method of alternating projections consists in choosing the initial low-rank matrix $Y_0 \in \mathcal{M}_r$ and computing successive Euclidean projections onto the two sets:
\begin{equation*}
    Z_{k+1} = \Proj{\mathbb{B}_\varepsilon(X)}{Y_k}, \quad Y_{k+1} \in \Proj{\mathcal{M}_r}{Z_{k+1}}, \quad k \geq 0.
\end{equation*}
The crucial aspect making alternating projections viable for our problem is that the projections are required to minimize the Euclidean distance (i.e. the Frobenius norm) rather than the maximum-norm distance:
\begin{equation*}
    \Proj{\mathcal{A}}{Y} = \set[\Big]{\hat{A} \in \mathcal{A}}{\Norm{\hat{A} - Y}{F} = \inf_{A \in \mathcal{A}} \Norm{A - Y}{F}}, \quad \mathcal{A} \subset \Real^{m \times n}.
\end{equation*}
For the closed ball $\mathbb{B}_\varepsilon(X)$, the Euclidean projection is unique, can be computed as
\begin{equation*}
    \Proj{\mathbb{B}_\varepsilon(X)}{Y_k} = X + \Proj{\mathbb{B}_\varepsilon(0)}{Y_k - X},
\end{equation*}
and amounts to clipping the large elements of $Y_k - X$. The Euclidean projection onto $\mathcal{M}_r$ follows from the truncated singular value decomposition and can be non-unique if there are repeated singular values (any truncation can be chosen in this case). When the initial condition $Y_0$ is ``not too bad'', both sequences $\{ Z_k \}$ and $\{ Y_k\}$ approach the same limit in $\mathcal{M}_r \cap \mathbb{B}_\varepsilon(X)$. Combining the method of alternating projections with a binary search over feasible values of $\varepsilon$, we can estimate the maximum-norm distance $d_r(X)$ and provide a numerical upper bound $d_r(X) \leq \varepsilon_{+}$. To accelerate the computations, we use randomized singular value decomposition for rank truncation. Find further details in \cite[Section~7.5]{budzinskiy2023quasioptimal}.

We are going to consider four different classes of matrices in order to better understand how the maximum-norm distance $d_r(X)$ depends on the properties of the matrix $X$.

\subsection{Identity matrices}

The first class are identity matrices $\{ I_n \}$ of varying sizes. Some preliminary numerical results on the identity matrices were reported in \cite{budzinskiy2023quasioptimal}, and our goal here is to elaborate on them. The maximum and spectral norms of $I_n$ are independent of $n$,
\begin{equation*}
    \Norm{I_n}{max} = 1, \quad \Norm{I_n}{2} = 1,
\end{equation*}
hence the ultimate upper bound \eqref{eq:ultimate_bound} and Theorem~\ref{theorem:our2_max_appr} guarantee, respectively, that $d_r(I_n) \leq 1$ for every $r$ and $d_{r_n}(I_n) \leq \varepsilon$ for $r_n = C \log(4n^2) / \varepsilon^2$. In two series of experiments, we use the method of alternating projections to numerically estimate $d_r(I_n)$ for
\begin{itemize}
    \item fixed matrix size $n$ and varying approximation rank $r$,
    \item fixed approximation rank $r$ and varying matrix size $n$.
\end{itemize}
For every pair $(n,r)$, we form a random initial matrix $Y_0 \in \mathcal{M}_r$ as a product of two $n \times r$ random Gaussian matrices with variance $r^{-1}$ and repeat the experiment five times with different instances of $Y_0$. We plot the smallest errors achieved over the five experiments in Fig.~\ref{fig:identity}.

\begin{figure}[h!]
\centering
	\includegraphics[width=\textwidth]{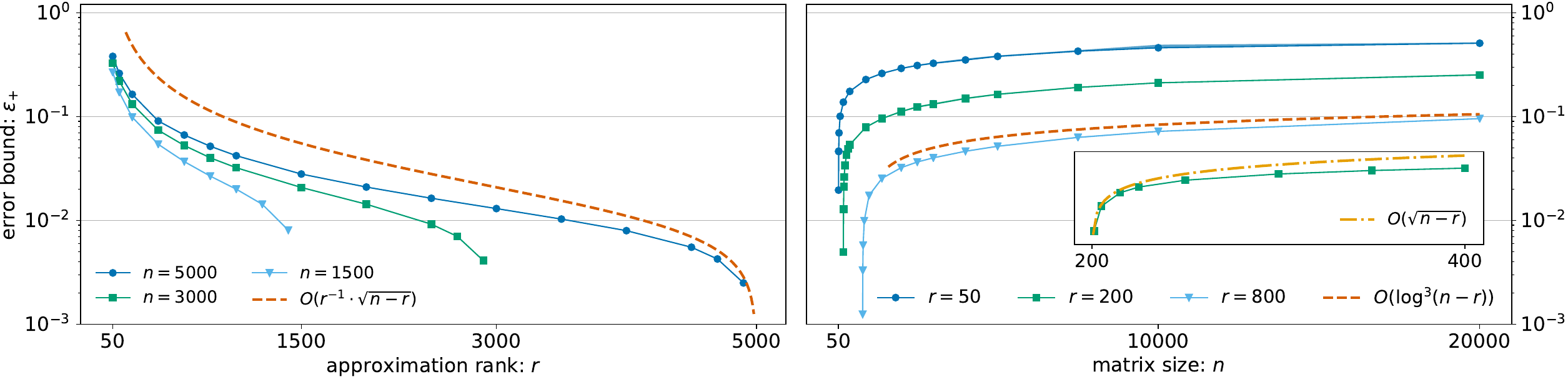}
\caption{Low-rank approximation errors in the maximum norm obtained with the method of alternating projections for the class of identity matrices: (left)~fixed matrix size $n$, (right)~fixed approximation rank $r$.}
\label{fig:identity}
\end{figure}

When $n$ is fixed and $r$ is small, we observe that the numerical estimate of $d_r(I_n)$ decays as $r^{-1}$ (same behavior was reported in \cite{budzinskiy2023quasioptimal} for $n = 1600$). As the approximation rank $r$ approaches $n$, the numerical error drops faster (to be expected) at a rate similar to $r^{-1} \sqrt{n - r}$. For moderate values of $r$ --- not too small and not too close to $n$ --- this numerical bound is in accordance with the $r^{-1/2}$ decay rate of Theorem~\ref{theorem:our2_max_appr}.

With fixed $r$, the error grows approximately as $\sqrt{n-r}$ for $n$ close to $r$ and slows down to the polylogarithmic rate for larger $n$ (we observed similar behavior for $r = 40$ in \cite{budzinskiy2023quasioptimal}, but with a different power of the logarithm). At the same time, Theorem~\ref{theorem:our2_max_appr} suggests that the approximation error grows as $\sqrt{\log(n)}$, which is slower than what we observe for the studied range of matrix sizes $n$.

\subsection{Random uniform matrices}

The second class are $n \times n$ random matrices $\{ U_n \}$ with independent entries distributed uniformly in the interval $(-1, 1)$. Their maximum and spectral norms with high probability satisfy the following \cite[Section~2.3]{tao2012topics}:
\begin{equation*}
    \Norm{U_n}{max} \approx 1, \quad \Norm{U_n}{2} = \mathcal{O}(\sqrt{n}).
\end{equation*}
This means that the upper bound of Theorem~\ref{theorem:our2_max_appr} for random uniform matrices is about $\sqrt{n}$ times larger than for identity matrices. We carry out two similar series of experiments as with the identity matrices, repeating the computations ten times with different instances of the initial approximation and plotting the median together with the 10\% and 25\% percentiles (see Fig.~\ref{fig:uniform}).

\begin{figure}[ht!]
\centering
	\includegraphics[width=\textwidth]{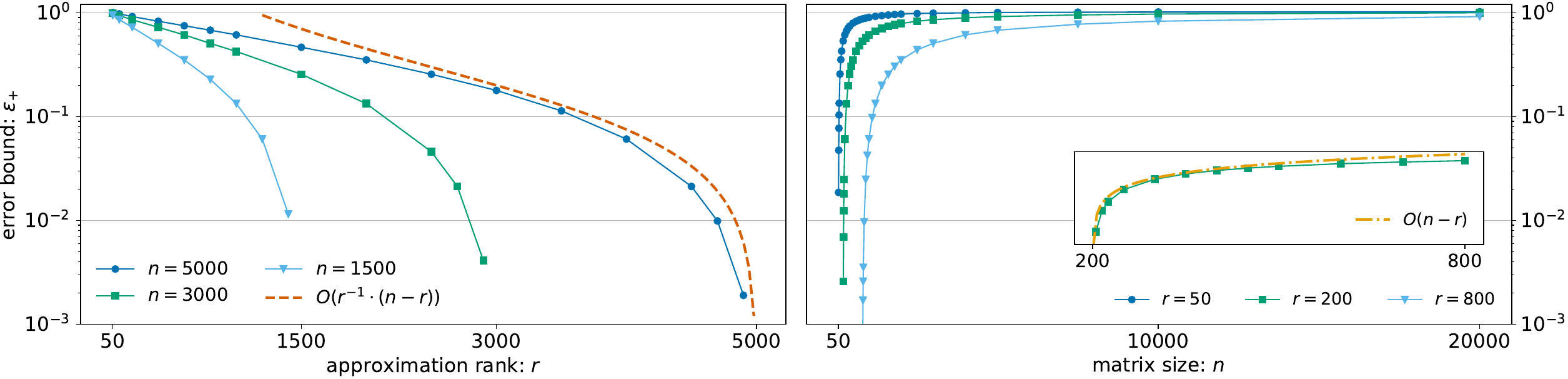}
\caption{Low-rank approximation errors in the maximum norm obtained with the method of alternating projections for the class of random uniform matrices: (left)~fixed matrix size $n$, (right)~fixed approximation rank $r$.}
\label{fig:uniform}
\end{figure}

For fixed matrix size $n$, our results show that the errors decrease approximately as $r^{-1} (n - r)$ when the approximation rank $r$ approaches $n$; this rate is higher than for identity matrices by a factor of $\sqrt{n - r}$. In addition, we do not see the $r^{-1}$-decay regime for small values of $r$, and the corresponding errors are close to one.

When we let $n$ grow and keep $r$ fixed, we observe that the errors increase much faster than for identity matrices and converge to the ultimate upper bound \eqref{eq:ultimate_bound} rapidly. The growth rate of the errors for small $n$ is similar to $n-r$, which confirms the speedup factor of $\sqrt{n - r}$ compared to identity matrices.

\subsection{Banded random uniform matrices}

The third class are banded random uniform matrices $\{ U_{n,b} \}$. An $n \times n$ matrix $U_{n,b}$ with the band-width parameter $b$, $1 \leq b \leq n$, has $2b-1$ non-zero central diagonals that consist of independent random entries distributed uniformly in $(-1,1)$, while all the remaining entries of $U_{n,b}$ are zero. The theoretical results on banded random matrices \cite{bandeira2016sharp} relate their expected spectral norm with $\sqrt{b}$. In Fig.~\ref{fig:banded_uniform}, we plot the median with the 10\% and 25\% percentiles of the spectral norms $\Norm{U_{n,b}}{2}$ for $n = 2000$ and varying $b$. The numerical estimate suggests that
\begin{equation*}
    \Norm{U_{n,b}}{max} \approx 1, \quad \Norm{U_{n,b}}{2} = \mathcal{O}(\sqrt{b}).
\end{equation*}

\begin{figure}[h!]
\centering
	\includegraphics[width=\textwidth]{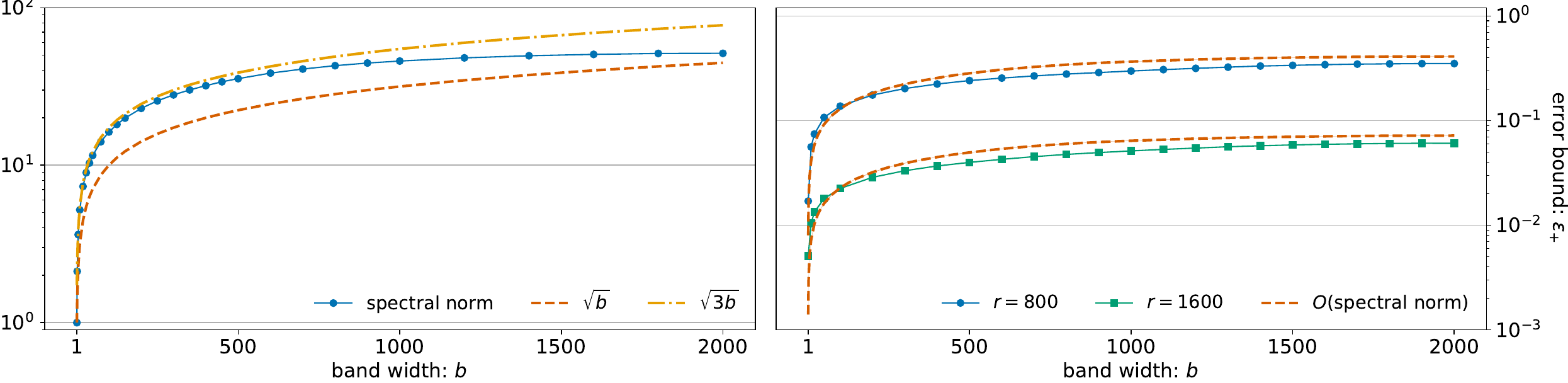}
\caption{Properties of $2000 \times 2000$ banded random uniform matrices: (left)~spectral norm as a function of band width $b$; (right)~low-rank approximation errors in the maximum norm obtained with the method of alternating projections for fixed approximation rank $r$.}
\label{fig:banded_uniform}
\end{figure}

By changing the band-width parameter $b$, we can control the spectral norm of the banded random matrix. This property allows us to investigate how the distance $d_r(U_{n,b})$ depends on $\Norm{U_{n,b}}{2}$ and verify if the relation is linear as suggested by Theorem~\ref{theorem:our2_max_appr}. We perform the experiments in the same settings as used with the random uniform matrices. The numerical results in Fig.~\ref{fig:banded_uniform} show that the growth of the distance is in perfect alignment with the spectral norm $\Norm{U_{n,b}}{2}$, confirming the way in which the spectral norm enters the estimate in Theorem~\ref{theorem:our2_max_appr} at least for a specific class of matrices.

\subsection{Products of random matrices with orthonormal columns}

The fourth class are matrices with incoherent low-dimensional column and row spaces $\{ P_{n,k} \}$. We form $n \times n$ matrices $P_{n,k}$ of rank $k$, $1 \leq k \leq n$, as products of two $n \times k$ random matrices with orthonormal columns that are drawn from the uniform distribution on the Stiefel manifold. Such low-rank factors can be computed by orthogonalizing $n \times k$ standard Gaussian random matrices. Upon normalizing the entries of $P_{n,k}$, we obtain from the numerical simulations (see Fig.~\ref{fig:lr_randort}) that the norms of these matrices behave as
\begin{equation*}
    \Norm{P_{n,k}}{max} = 1, \quad \Norm{P_{n,k}}{2} = \mathcal{O}(k^{-0.55}).
\end{equation*}

\begin{figure}[h!]
\centering
	\includegraphics[width=\textwidth]{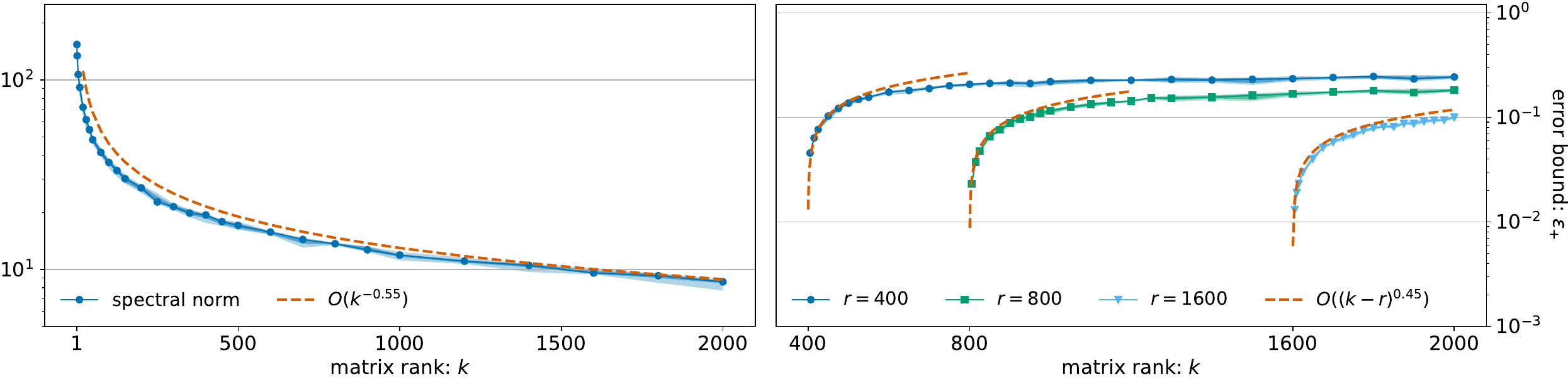}
\caption{Properties of $2000 \times 2000$ matrices formed as products of $2000 \times k$ random matrices with orthonormal columns and normalized to unit maximum norm: (left)~spectral norm as a function of rank $k$; (right)~low-rank approximation errors in the maximum norm obtained with the method of alternating projections for fixed approximation rank $r$.}
\label{fig:lr_randort}
\end{figure}

The numerical experiments in Fig.~\ref{fig:lr_randort} suggest that the distance $d_r(P_{n,k})$ grows approximately as $(k-r)^{0.45}$ when the approximation rank $r$ is close to the actual rank $k$ and then stabilizes as $k$ approaches $n$, its maximum possible value. At the same time, the upper bound of Theorem~\ref{theorem:our2_max_appr} scales with $k$ as $k \Norm{P_{n,k}}{2} = \mathcal{O}(k^{0.45})$ --- the same power law as we observe in the experiments.
\section{Discussion}

We can formulate three implications from the results of our numerical experiments. First, they illustrate how the distance $d_r(X)$ depends on $\rank{X}$ and confirm the corresponding scaling in Theorem~\ref{theorem:our2_max_appr}. The proof of such dependency is a novel contribution of our article as neither Theorem~\ref{theorem:spd_max_appr} nor Theorem~\ref{theorem:ut_max_appr} describe it.

Second, the numerical results serve as evidence that, for general matrices $X$, the distance $d_r(X)$ depends on the spectral norm $\Norm{X}{2}$ (as suggested by Theorems~\ref{theorem:ut_max_appr} and \ref{theorem:our2_max_appr}) rather than on the maximum norm $\Norm{X}{max}$ (as in the symmetric positive semidefinite case of Theorem~\ref{theorem:spd_max_appr}). Therefore, it is likely that the need to replace $\Norm{X}{max}$ with $\Norm{X}{2}$ in the process of generalizing Theorem~\ref{theorem:spd_max_appr} to non-symmetric matrices is not a technical limitation of the proof.

Third, our results reveal a phase transition. When the approximation rank $r$ and the actual $\rank{X}$ are close, the distance $d_r(X)$ exhibits power-law scaling with respect to parameters. Whereas when $r$ is small compared to $\rank{X}$, the rate of change switches to the polylogarithmic regime, until $d_r(X)$ eventually approaches the ultimate upper bound \eqref{eq:ultimate_bound}. 

It would be interesting to see our numerical experiments repeated --- and extended --- with different algorithms, especially with those that have stronger optimality guarantees.

\printbibliography

\end{document}